\documentclass[10pt]{amsart}

\usepackage[latin1]{inputenc}
\usepackage{amsmath}
\usepackage{amsfonts}
\usepackage{amssymb}
\usepackage{ytableau}
\usepackage[mathscr]{eucal}
\usepackage{diagrams}
\usepackage{xy,color,graphicx,tikz-cd}
\usepackage{hyperref}
\usepackage{enumitem}
\usepackage{microtype}
\xyoption{all}

\newcommand{\wis}[1]{{\text{\em \usefont{OT1}{cmtt}{m}{n} #1}}}

\newcommand{\N}{\mathbb{N}}

\newcommand{\C}{\mathbb{C}}

\newcommand{\Oscr}{\mathcal{O}}

\DeclareMathOperator{\lcm}{lcm}

\newtheorem{definition}{Definition}[section]
\newtheorem{proposition}[definition]{Proposition}

\newtheorem{theorem}[definition]{Theorem}

\newtheorem{lemma}[definition]{Lemma}
\newtheorem{example}[definition]{Example}

\tikzset{
b/.style={bend left=10},
bb/.style={bend left},
cl/.style={outer sep=-1pt},
}

\title{Azumaya geometry and representation stacks }
\author{Jens Hemelaer} 
\address{Department of Mathematics, University of Antwerp \\ 
 Middelheimlaan 1, B-2020 Antwerp (Belgium) \\ {\tt jens.hemelaer@uantwerpen.be}}
\thanks{Jens Hemelaer is Ph.D.\ fellow of the Research Foundation -- Flanders (FWO)}
\author{Lieven Le Bruyn} 
\address{Department of Mathematics, University of Antwerp \\ 
 Middelheimlaan 1, B-2020 Antwerp (Belgium) \\ {\tt lieven.lebruyn@uantwerpen.be}}

\begin{document}
 
 \begin{abstract} We develop Azumaya geometry, which is an extension of classical affine geometry to the world of Azumaya algebras, and package the information contained in all quotient stacks $[\wis{rep}_n R/\mathrm{PGL}_n]$ into a presheaf $\wis{Rep}_R$ on it. We show that the classical \'etale and Zariski topologies extend to Grothendieck topologies on Azumaya geometry in uncountably many ways, and prove that $\wis{Rep}_R$ is a sheaf for all of them. The restriction to a specific Azumaya algebra $A$ with center $C$ gives us a sheaf in the \'etale topology which is represented by an affine $C$-scheme $\wis{rep}_A(R)$, which we call the Azumaya representation scheme of $R$ with respect to $A$.
 \end{abstract}
 
  \maketitle

 \section{Introduction}
 
 For a finitely generated $\C$-algebra $R$ one can attempt to classify its $n$-dimensional representation geometrically by investigating the $\mathrm{PGL}_n$-orbits of the representation scheme $\wis{rep}_n~R$. As the GIT-quotient scheme $\wis{rep}_n~R/\mathrm{PGL}_n$ only classifies isomorphism classes of semi-simple representations one needs to consider the quotient stack $[\wis{rep}_n~R/\mathrm{PGL}_n]$, called the {\em representation stack}, whose $\C$-points classify all $n$-dimensional representations. In \cite{LBstacks} a ring theoretical interpretation was given of the $C$-points of this representation stack, for any commutative $\C$-algebra $C$. They correspond to $\C$-algebra morphisms $R \rTo A$ where $A$ is a degree $n$ Azumaya algebra with center $C$.
 
Let $\wis{Azumaya}_{\C}$ be the category of all Azumaya algebras with center an affine $\C$-algebra and center-preserving ring morphisms, then all information from the representation stacks $[\wis{rep}_n~R/\mathrm{PGL}_n]$, for all $n$, is contained in the covariant functor
 \[
 \wis{Rep}_R~:~\wis{Azumaya}_{\C} \rTo \wis{sets} \qquad A \mapsto \mathrm{Alg}_{\C}(R,A) \]
 Assigning to each Azumaya algebra its non-commutative affine scheme, as in \cite{FVOVerschoren}, we will develop Azumaya affine geometry, $\wis{Azufine}_{\C} = \wis{Azumaya}^{op}$, a seemingly innocuous extension of classical affine geometry into the non-commutative realm. Some marked differences are that there do not exist pullbacks in Azumaya geometry and that \'etale and Zariski topologies extend to Grothendieck topologies on Azumaya geometry in uncountably many different ways. In Theorem~\ref{main0} we classify these Grothendieck topologies via patches on the supernatural numbers, see also \cite{Hemelaer}.
 
 We will prove that the presheaf $\wis{Rep}_R$ on Azumaya geometry is a sheaf for each of these Grothendieck topologies, see Theorem~\ref{main1}. If we fix an Azumaya algebra $A$ with center $C$, the restriction of this functor to affine $C$-schemes gives us a sheaf in the \'etale and Zariski topology
 \[
 \wis{Rep}_A(R)~:~\wis{affine}_C \rTo \wis{sets} \qquad Z \mapsto \mathrm{Alg}_{\C}(R,A \otimes_C \Oscr(Z)) \]
 and in Theorem~\ref{main2} we show that it is representable by an affine $C$-scheme $\wis{rep}_A(R)$, which we call the Azumaya representation scheme of $R$ with respect to $A$, with $C$-points $\wis{Rep}_R(A)$.

 \section{Representation stacks}

 Throughout, all algebras will be associative, unital $\C$-algebras which are finitely generated, that is, have a presentation
\[
R = \frac{\C \langle x_1,\hdots,x_k \rangle}{(p_i(x_1,\hdots,x_k)~|~ i \in I)} \]
In order to study the isomorphism classes of finite dimensional representations of $R$ one introduces the affine schemes $\wis{rep}_n~R$, for all $n$, with coordinate ring the quotient of the polynomial algebra in all entries of $k$ generic $n \times n$ matrices
\[
X_l = \begin{bmatrix} x_{11}(l) & \hdots & x_{1n}(l) \\
\vdots & & \vdots \\
x_{n1}(l) & \hdots & x_{nn}(l) \end{bmatrix} \]
modulo the ideal generated by all the matrix entries of the $(n \times n)$-matrices $p_i(X_1,\hdots,X_k)$ for all $i \in I$. That is,
\[
\Oscr(\wis{rep}_n~R) = \frac{\C[x_{ij}(l)~:~1 \leq i,j \leq n, 1 \leq l \leq k]}{(p_i(X_1,\hdots,X_k)_{uv}~:~1 \leq u,v \leq n,~i \in I)} \]
The affine group $\mathrm{GL}_n$ (or rather $\mathrm{PGL}_n$) acts via simultaneous conjugation on the generic matrices $X_i$ and hence on the affine scheme $\wis{rep}_n~R$. Its orbits on the $\C$-points of $\wis{rep}_n~R$ correspond to isomorphism classes of $n$-dimensional representations of $R$. 

From Mumford's geometric invariant theory we know that the $\C$-points of the quotient scheme $\wis{rep}_n R/\mathrm{PGL}_n$, corresponding to the ring of polynomial invariants $\Oscr(\wis{rep}_n R)^{\mathrm{PGL}_n}$, classify isomorphism classes of {\em closed orbits}. A result of M. Artin \cite{Artin1969} asserts that closed $\mathrm{PGL}_n$-orbits in $\wis{rep}_n~R$ are in one-to-one correspondence with isomorphism classes of {\em semi-simple} $n$-dimensional representations of $R$.

\vskip 3mm

In order to study also the non-semi-simple $n$-dimensional representations one has to extend the framework of affine schemes and consider the {\em quotient stack}, see \cite[\href{https://stacks.math.columbia.edu/tag/04SL}{Chapter 04SL}]{stacks-project}, $[\wis{rep}_n~R/\mathrm{PGL}_n]$, which we call the $n$-{\em representation stack} of $R$. By the general recipe of \cite[\href{https://stacks.math.columbia.edu/tag/04SL}{Chapter 04SL}]{stacks-project} this representation stack corresponds to the functor
\[
[\wis{rep}_n~R/\mathrm{PGL}_n]~:~\wis{affine}_{\C} \rTo \wis{groupoids} \qquad Z \mapsto [\wis{rep}_n~R/\mathrm{PGL}_n](Z) \]
from the category $\wis{affine}_{\C}$ of all affine $\C$-schemes of finite presentation to the category $\wis{groupoids}$ of all groupoids, that is categories in which every morphism is an isomorphism. This functor associates to an affine $\C$-scheme $Z$ the category $[\wis{rep}_n~R/\mathrm{PGL}_n](Z)$ with  objects all triples $(Z,P,\phi_P)$ with $\pi_P : P \rOnto Z$ a principal $\mathrm{PGL}_n$-fibration in the \'etale topology over $Z$ and $\phi_P$ a $\mathrm{PGL}_n$-equivariant map.
\[
\xymatrix{
P \ar[rr]^{\phi_P} \ar[d]^{\pi_P} & & \wis{rep}_n R \\
Z & & }
\]
and morphisms in this category consist of $\mathrm{PGL}_n$-equivariant maps $h$ making the diagram below commutative
\[
\xymatrix{& \wis{rep}_n~R & \\
P \ar[ru]^{\phi_P} \ar[rd]_{\pi_P} \ar[rr]^h & & P' \ar[lu]_{\phi_{P'}} \ar[ld]^{\pi_{P'}} \\
& Z & } \]
One easily verifies that $[\wis{rep}_n~R/\mathrm{PGL}_n](Z)$ is indeed a groupoid, called the {\em $Z$-points} of the quotient stack $[\wis{rep}_n~R/\mathrm{PGL}_n]$.

In \cite[Theorem 1]{LBstacks} a purely ring theoretical interpretation was given of the groupoid $[\wis{rep}_n~R/\mathrm{PGL}_n](Z)$. Recall that an $\Oscr(Z)$-algebra $A$ is said to be an {\em Azumaya algebra of degree $n$ over $\Oscr(Z)$} provided $A$ is a projective $\Oscr(Z)$-module of constant rank $n^2$ such that the natural $\Oscr(Z)$-algebra morphism $j : A \otimes_{\Oscr(Z)} A^{op} \rTo End_{\Oscr(Z)}(A)$ given by $j(a \otimes b)(r)=arb$ is an isomorphism. It follows that the center $Z(A)$ of $A$ is equal to $\Oscr(Z)$ and that $A/\mathfrak{m}A \simeq \mathrm{M}_n(\C)$ for every maximal ideal $\mathfrak{m}$ of $\Oscr(Z)$. That is, every $n$-dimensional representation of $A$ is simple and $\wis{rep}_n~A$ is a principal $\mathrm{PGL}_n$-fibration in the \'etale topology with quotient scheme $\wis{rep}_n~A/\mathrm{PGL}_n \simeq Z$, see for example \cite{KnusOjanguren}.

Therefore, any triple $(Z,P,\phi_P)$ in $[\wis{rep}_n~R/\mathrm{PGL}_n]$ corresponds uniquely to a pair $(A,\phi)$ where $A$ is an $\Oscr(Z)$-Azumaya algebra of degree $n$ and $\phi : R \rTo A$ is a $\C$-algebra morphism. Here, $P = \wis{rep}_n A$ and
\[
\phi_P~:~P=\wis{rep}_n~A \rTo \wis{rep}_n~R \]
is the $\mathrm{PGL}_n$-equivariant map induced (by composition) from the algebra map $\phi$. The fact that this category is a groupoid now translates into the classical result that any $\Oscr(Z)$-algebra morphism $\psi : A \rTo A'$ between two degree $n$ Azumaya algebras over $\Oscr(Z)$ is an isomorphism. 

\section{Azumaya geometry}
 
 Motivated by the description of representation stacks of $R$ in the previous section we will introduce in this section a seemingly innocuous extension of classical affine geometry $\wis{affine}_{\C}$ to the non-commutative realm.  
 
 For an affine scheme $Z$ and a degree $n$ Azumaya algebra $A$ over $\Oscr(Z)$ we know that the map $P \mapsto P \cap \Oscr(Z)$ gives a homeomorphism
 \[
 j_A~:~\mathrm{Spec}(A) \rTo Z \]
 between the prime spectrum $\mathrm{Spec}(A)$ of $A$, as in \cite{ProcesiBook}, and the prime spectrum of $\Oscr(Z)$ which is $Z$. As every localization of $A$ is a central localization, it follows that the noncommutative structure sheaf of $A$ over $\mathrm{Spec}(A)$, as in \cite{FVOVerschoren}, coincides with the pullback $j_A^* \Oscr_A$ of the structure sheaf $\Oscr_A$ over $Z$ of the $\Oscr(Z)$-algebra $A$. We will call the ringed space $\mathcal{A}=(\mathrm{Spec}(A),j_A^* \Oscr_A)$ the {\em non-commutative affine scheme} with coordinate ring $A$.
 
 With $\wis{Azumaya}_{\C}$ we will denote the category with objects all Azumaya algebras $A$ with center $Z(A)=\Oscr(Z)$, the coordinate ring of an affine $\C$-scheme $Z$ of finite presentation, and with morphisms $f : A \rTo A'$ all {\em center preserving} $\C$-algebra morphisms, that is we impose that $f(Z(A)) \subseteq Z(A')$.
 
 As any center preserving algebra morphism $j : A \rTo A'$ between Azumaya algebras is a {\em central extension} in the sense of C. Procesi, it follows from \cite{ProcesiBook} and \cite{FVOVerschoren} that it determines (and is determined by) a unique morphism of ringed spaces
 \[
 \mathcal{A}'=(\mathrm{Spec}(A'),j_{A'}^* \Oscr_{A'}) \rTo (\mathrm{Spec}(A),j_A^* \Oscr_A) = \mathcal{A} \]
 For this reason we can define {\em Azumaya affine geometry} as the opposite category of $\wis{Azumaya}_{\C}$
 \[
 \wis{Azufine}_{\C} \simeq (\wis{Azumaya}_{\C})^{op} \]
 Similar to how one obtains all schemes from gluing affine schemes, one can develop from these non-commutative affine Azumaya schemes a more general Azumaya geometry. Perhaps it is worth mentioning that moduli spaces of stable representations of finite dimensional $\C$-algebras come naturally equipped with a sheaf of Azumaya algebras and hence are objects in Azumaya geometry.
 
 \vskip 3mm
 
 Returning to the problem of classifying all finite dimensional representations of the non-commutative algebra $R$ geometrically, we see that all information contained in the representation stacks $[\wis{rep}_n~R/\mathrm{PGL}_n]$, for all $n$, is contained in the contravariant functor
 \[
 \wis{Rep}_R~:~\wis{Azufine}_{\C} \rTo \wis{sets} \qquad \mathcal{A} \mapsto \mathrm{Alg}_{\C}(R,A) \]
 which assigns to the affine Azumaya scheme $\mathcal{A}=(\mathrm{Spec}(A),j_A^* \Oscr_A)$ the set $\mathrm{Alg}_{\C}(R,A)$ of all $\C$-algebra morphisms from $R$ to $A$. In other words, $\wis{Rep}_R$ is an object in the {\em presheaf topos} on $\wis{Azufine}_{\C}$, see \cite{MM}. It therefore makes sense to investigate those Grothendieck topologies on $\wis{Azufine}_{\C}$ for which $\wis{Rep}_R$ is a sheaf.
 
 In classical affine geometry, apart from the Zariski topology also the \'etale topology is important with respect to Azumaya algebras as any Azumaya algebra is split by an \'etale cover. In the next section we will determine natural extensions of these two Grothendieck topologies on $\wis{affine}_{\C}$ to $\wis{Azufine}_{\C}$. 
 
 On $\wis{affine}_{\C}$ the classification of Grothendieck topologies is substantially simplified because we can reduce to a basis for the topology (see \cite[Def. III.2.2]{MM}), due to the fact that $\wis{affine}_{\C}$ has pullbacks, or equivalently, that commutative affine $\C$-algebras have pushouts (aka tensor products). We will now see that, although $\wis{Azumaya}_{\C}$ does allow for tensor products, these are not necessarily pushouts. So, in order to determine Grothendieck topologies on $\wis{Azufine}_{\C}$ we will have no other option than to use {\em sieves}.
 
 \begin{lemma} $\wis{Azumaya}_{\C}$ has tensor products but not pushouts. As a consequence $\wis{Azufine}_{\C}$ does not have pullbacks.
 \end{lemma}
 
 \begin{proof} If $f_i : A \rTo A_i$ are morphisms in $\wis{Azumaya}_{\C}$, then $A \otimes_{Z(A)} Z(A_i)$ is an $Z(A_i)$-Azumaya subalgebra of $A_i$, whence b the double centralizer theorem \cite[Thm. II.4.3]{DeMeyerIngraham} we have
 \[
 A_i \simeq (A \otimes_{Z(A)} Z(A_i)) \otimes_{Z(A_i)} A_i^A \simeq A \otimes_{Z(A)} A_i^A \]
 where $A_i^A = \{ a \in A_i~|~a \in Z(f_i(A)) \}$. But then,
 \[
 A_1 \otimes_A A_2 \simeq A_1^A \otimes_{Z(A)} A \otimes_A A \otimes_{Z(A)} A_2^A \simeq A_1^A \otimes_{Z(A)} A \otimes_{Z(A)} A_2^A \]
 which is isomorphic to $A_1 \otimes_{Z(A)} A_2^A$ and as $A_1$ and $A_2^A$ are separable over $Z(A)$ so is their tensor product, whence an Azumaya algebra over its center.
 
 However, tensor products are not pushouts in $\wis{Azumaya}_{\C}$. For example take an Azumaya algebra $A$ of degree $n > 1$ over $Z(A)$, then $A \otimes_{Z(A)} A$ s a degree $n^2$ Azumaya algebra over $Z(A)$ and hence does not allow for a $Z(A)$-algebra morphism $A \otimes_{Z(A)} A \rTo A$
 as would be required if it were the pushout when we consider the pair of morphisms $(id : A \rTo A, \alpha : A \rTo A)$ for $\alpha$ a $Z(A)$-algebra automorphism of $A$ (e.g. $\alpha=id$). 
  \end{proof}

 \section{Grothendieck topologies on \texorpdfstring{$\wis{Azufine}_{\C}$}{Azufine(C)}}
 
 Because $\wis{Azufine}_{\C}$ does not have pullbacks we will have to define {\em Grothendieck topologies} via {\em sieves} as in \cite[III.2]{MM}. As we prefer to retain an algebraic description we will work with the dual notions in $\wis{Azumaya}_{\C}$.
 
 \begin{definition} A {\em sieve} $S$ on an Azumaya algebra $A$ is a collection $S = \{ A \rTo^f A_f \}$ of morphisms in $\wis{Azumaya}_{\C}$ closed under post-composition. That is, for $f \in S$ and any $g : A_f \rTo B$ in $\wis{Azumaya}_{\C}$ also $g \circ f : A \rTo B \in S$.
 \end{definition}
 
 \begin{definition} \label{def:grothendieck-topology}
 A {\em Grothendieck topology} on $\wis{Azumaya}_{\C}$ is a function $J$ which assigns to each Azumaya algebra $A$ a collection $J(A)$ of sieves on $A$ satisfying the following properties
\begin{enumerate}
\item[(GT1)]{The maximal sieve $T_A = \{ f : A \rTo B \in \wis{Azumaya}_{\C} \}$ of all morphisms from $A$ is an element of $J(A)$}
\item[(GT2)]{Stability: If $S \in J(A)$, then for any morphism $h : A \rTo B$ in $\wis{Azumaya}_{\C}$, $h^{-1}(S)=\{ g: B \rTo D~:~g \circ h \in S\} \in J(B)$}
\item[(GT3)]{Transitivity: If $S \in J(A)$ and $R$ is a sieve on $A$ such that $h^{-1}(R) \in J(B)$ for all morphisms $h : A \rTo B$ in $S$, then $R \in J(A)$.}
\end{enumerate}
We say that a collection of morphisms $\{ A \rTo A_i \}_{i \in I}$ is a {\em cover} of $A$ with respect to the Grothendieck topology $J$ if these morphisms generate (by post-composition) a sieve in $J(A)$. 
\end{definition}

We will first give a combinatorial description of sieves and Grothendieck topologies on the
full subcategory $\wis{Mat}_{\C}$ of $\wis{Azumaya}_{\C}$ consisting of the matrix algebras $\mathrm{M}_n(\C)$ for all $n \in \mathbb{N}_+$. Let $\mathbb{N}_+^{\times}$ be the poset category with objects $n \in \mathbb{N}_+$ and morphisms $n \rTo m$ iff $n | m$. We have a projection functor $\pi : \wis{Mat}_{\C} \rTo \mathbb{N}^{\times}_+$ sending a morphism $\mathrm{M}_n(\C) \rTo \mathrm{M}_{nk}(\C)$ to $n \rTo nk$.

\begin{lemma} 
Sieves on $\mathrm{M}_n(\C)$ in $\wis{Mat}_{\C}$ are in bijection with sieves on $n$ in $\mathbb{N}^{\times}_+$ via $S \mapsto \pi(S)$. As a consequence, Grothendieck topologies on $\wis{Mat}_{\C}$ are in bijection with Grothendieck topologies on $\mathbb{N}^{\times}_+$.\end{lemma}

\begin{proof} The result follows if we can show that a sieve $S$ on $\mathrm{M}_n(\C)$ is fully determined by the multiples of $n$ such that there is a morphism $\alpha : \mathrm{M}_n(\C) \rTo \mathrm{M}_{nk}(\C) \in S$ and not on the actual morphism $\alpha$. So, let $\beta : \mathrm{M}_n(\C) \rTo \mathrm{M}_{nk}(\C)$ be another morphism, then it follows from the double centralizer theorem that there is an automorphism $\gamma$ of $\mathrm{M}_{nk}(\C)$ such that $\gamma \circ \alpha = \beta$. But then we have
\[
\alpha \in S \Leftrightarrow \beta \in S. \]
\end{proof}

Grothendieck topologies on $\mathbb{N}^{\times}_+$ have been studied in \cite{LBas2} in connection with the Arithmetic Site of Connes and Consani \cite{CC}. Sieves $S$ on $n$ correspond one-to-one with ideals $M= \cup_i n_i \mathbb{N}_+$ of the multiplicative monoid $\mathbb{N}^{\times}_+$ via $n \rTo nk \in S$ iff $k \in M$. Further, if $h : n \rTo n'$ is a morphism in $\mathbb{N}^{\times}_+$ and a sieve $S$ corresponds to $M = \cup_i n_i \mathbb{N}_+$, then $h^{-1}(S)$ corresponds to $\cup_i \lcm(n',n_i) \mathbb{N}_+$. 

For our purposes Grothendieck topologies {\em of finite type} on $\mathbb{N}^{\times}_+$ will be important, that is, those for which every covering sieve contains a finitely generated covering sieve. These have been classified in \cite{Hemelaer} to be in one-to-one correspondence with {\em patches} in $\mathbb{S}$, the multiplicative monoid of {\em supernatural numbers}. Recall that $\mathbb{S}$ consists of all formal products $s = \prod_p p^{e_p}$ over all primes $p$ with each $e_p \in \N \cup \{ \infty \}$ with obvious multiplication.
There are uncountably many finite type Grothendieck topologies on $\mathbb{N}^{\times}_+$, see \cite[Example 2.7]{Hemelaer} for some classes of examples.

 \begin{example} \label{eg:sigma-topologies}
Consider a set $\Sigma$ of prime numbers. To this set we can associate a Grothendieck topology $K_{\Sigma}$ on $\mathbb{N}^{\times}_+$ by taking for the collection $K_{\Sigma}(n)$ of the sieves on $n$ all sieves corresponding to ideals $M = \cup_i n_i \mathbb{N}_+$ such that at least one $n_i$ has all its prime divisors in $\Sigma$. It is easy to check that $K_\Sigma$ defines a Grothendieck topology on $\mathbb{N}^{\times}_+$. Moreover $K_\Sigma = K_{\Sigma'}$ implies that $\Sigma = \Sigma'$, for two sets of primes $\Sigma$ and $\Sigma'$.

As the collection of ideals describing the sieves on elements is equal for all $n \in \mathbb{N}_+$, these Grothendieck topologies are {\em stable under multiplication}, that is they have the property that if $\{ n \rTo n_r \}_{r \in R}$ is in $K_{\Sigma}(n)$ and if $k$ is any positive integer, then $\{ nk \rTo n_r k \}_{r \in R} \in K_{\Sigma}(nk)$. If in particular $\Sigma = \mathbb{P}$ is the set of all prime numbers, then the corresponding Grothendieck topology $K_+ = K_{\mathbb{P}}$ will be called the {\em maximal topology} on $\mathbb{N}^{\times}_+$. In contrast, the topology $K_- = K_\emptyset$ corresponding to the empty set will be called the \emph{minimal topology} on $\mathbb{N}^\times_+$.
\end{example}

\vskip 3mm

Having determined suitable Grothendieck topologies $K$ on $\wis{Mat}_{\C}$, the other ingredient in defining Grothendieck topologies on $\wis{Azumaya}_{\C}$ will be a Grothendieck topology $J$ on $\wis{commalg}_{\C}=\wis{affine}_{\C}^{op}$, the category of all affine (i.e.~finitely presented) commutative $\C$-algebras. Such a $J$ assigns to each $C \in \wis{commalg}_{\C}$ a collection of covers $\{ C \rTo C_i \}_{i \in I}$.
Important examples are the {\em Zariski topology} on $\wis{commalg}_{\C}$ generated by the covers
\[
\{ C \rTo C[c_i^{-1}] \}_{i \in I} \quad\text{with}\quad\sum_{i \in I} c_iC = C, \]
the {\em \'etale topology} generated by the covers
\begin{gather*}
\{ C \rTo C[c_i^{-1}] \rTo D_i \}_{i \in I} \quad\text{with}\quad \sum_{i \in I} c_iC=C\quad \\ \text{and}\quad\mathrm{Spec}(D_i) \rTo \mathrm{Spec}(C[c_i^{-1}])~\text{\'etale surjections},
\end{gather*}
and the {\em flat topology} generated by the covers
\[
\{ C \rTo C_i \}_{i \in I} \quad\text{with}\quad C \rTo \prod_{i \in I} C_i~\text{faithfully flat},~ I \text{ finite}. \]
Note that Grothendieck topologies are sometimes defined on the category of \emph{all} commutative rings. Then there are two notions of flat topology (called fppf and fpqc in the literature). However, we will follow the approach of \cite{gabber-kelly}, by restricting to finitely presented rings, so the two notions coincide.

Now we are ready to construct Grothendieck topologies on $\wis{Azumaya}_{\C}$. Let $K$ be a Grothendieck topology on $\mathbb{N}^{\times}_+$. Let $A$ be an Azumaya algebra with center $C$, and take a sieve 
\begin{equation}
L = \{ A \to A_i \}_{i \in I}.
\end{equation}
Let $f : C \to D$ be a morphism of commutative rings. 

\begin{definition}
Take $K$, $L$ and $f$ be as above. Then we say that $m \in \mathbb{N}_+$ is \emph{represented} on $f$ if $L$ contains a central extension of $A \otimes_C D$ that is of constant degree $m$ over its center. If we can find such a central extension by a matrix algebra, then we say that $m$ is \emph{represented by a matrix algebra} on $f$. We say that $f$ is \emph{centrally covered} if $A \otimes_C D$ is of constant degree $n$, and the represented numbers on $f$ form a $K$-covering sieve on $n$. We now define:
\begin{equation}
\pi_K(L) = \{ f : C \to D \text{ such that } f \text{ is centrally covered w.r.t.\ }K\text{ and }L \}.
\end{equation}
Similarly, we say that $f$ is \emph{centrally covered by matrix algebras} if $A \otimes_C D$ is isomorphic to a matrix algebra of constant degree $n$, and the numbers that are represented by a matrix algebra on $f$, form a $K$-covering sieve on $n$. We then define:
\begin{equation}
\Pi_K(L) = \left\{ \substack{f ~:~ C ~\to~ D \text{ such that } f \text{ is centrally } \\ \text{covered by matrix algebras w.r.t.\ }K\text{ and }L} \right\}.
\end{equation}
\end{definition}

\begin{definition} \label{def:JK}
Consider $(J,K)$ with $J$ a Grothendieck topology on $\wis{commalg}_{\C}$ and $K$ a Grothendieck topology on $\mathbb{N}^{\times}_+$. Let $L = \{ A \to A_i \}_{i \in I}$ be a sieve on $A$. Then we say that $L$ is a $J_K$-covering sieve if $\pi_K(L)$ as above is a $J$-covering sieve.
\end{definition}

For example, take a sieve $S = \{ C \to C_i \}_{i \in I}$, and let $\tilde{S}$ be the sieve in $\wis{Azumaya}_{\C}$ generated by $S$. Suppose that every covering sieve in $K$ is nonempty. Then:
\begin{equation}
\Pi_K(\tilde{S}) ~=~ \pi_K(\tilde{S}) ~=~ S.
\end{equation}
In particular, $\tilde{S}$ is a $J_K$-covering sieve if and only if $S$ is a $J$-sieve. 

We now want to show that under certain conditions the collection of $J_K$-covering sieves is a Grothendieck topology on $\wis{Azumaya}_{\C}$.

\begin{theorem} \label{main0} Let $J$ be a Grothendieck topology on $\wis{commalg}_{\C}$ and $K$ a Grothendieck topology on $\mathbb{N}^{\times}_+$. Suppose that $K$ is of finite type. Further, suppose that
\begin{enumerate}
\item $J$ is finer than the \'etale topology, or
\item $J$ is finer than the Zariski topology and $K$ is stable under multiplication.
\end{enumerate}
Then the collection of $J_K$-covering sieves of Definition \ref{def:JK} defines a Grothendieck topology on $\wis{Azumaya}_{\C}$.
\end{theorem}
\begin{proof}[Proof] 
We prove the three axioms for a Grothendieck topology, see Definition \ref{def:grothendieck-topology}.

(GT1). If $A$ is Azumaya with center $C$, then the maximal sieve $L$ on $A$ is a $J_K$-covering sieve. Indeed, let $f : C \to D$ be any morphism of commutative rings such that $A \otimes_C D$ is of constant degree. Then $f$ is centrally covered, i.e.\ $f \in \pi_K(L)$. So $\pi_K(L)$ contains a Zariski covering, in particular it is a $J$-covering sieve.

(GT2). Let $L=\{ A \to A_i \}_{i \in I}$ be a $J_K$-covering sieve, and let $\phi : A \to A'$ be a morphism of Azumaya algebras, inducing a morphism $\phi_0 : C \to C'$ on centers. We need to show that $\phi^{-1} L$ is again a $J_K$-covering sieve. We treat the cases (1) and (2) from the theorem statement separately.

\underline{Case (1)}. We claim that $\Pi_K(L)$ is a $J$-covering sieve. It is enough to show that $f^{-1}\Pi_K(L)$ is a $J$-covering sieve, for each $f \in \pi_K(L)$. So take such a centrally covered morphism $f : C \to D$. The represented numbers on $f$ form a $K$-covering sieve, and we can take a finitely generated covering sieve $(m_1,\dots,m_k)$ contained in it, with each $m_i$ represented by an Azumaya algebras $A_i$ with center $D$. Take an \'etale cover $g:D \to E$ trivializing $A_1,\dots,A_k$. Then $g$ is contained in $f^{-1}\Pi_K(L)$, so $f^{-1}\Pi_K(L)$ is an \'etale covering sieve, so in particular it is a $J$-covering sieve. This shows the claim.

Now take a morphism $g : C' \to D$ such that $g \circ \phi_0$ is centrally covered by matrix algebras, and such that $A' \otimes_{C'} D$ is isomorphic to a matrix algebra. Then $g$ is itself centrally covered with respect to $\phi^{-1}L$. So there is an inclusion
\begin{equation}
M \cap \phi_0^{-1}\Pi_K(L) \subseteq \Pi_K(\phi^{-1}L)
\end{equation}
where $M$ is an \'etale covering sieve trivializing $A'$. This shows that $\Pi_K(\phi^{-1}L)$ is a $J$-covering sieve, so $\phi^{-1}L$ is a $J_K$-covering sieve.

\underline{Case (2)}. Let $M$ be a Zariski covering sieve on $C'$ such that for $g : C' \to D$ in $M$ we have that $A \otimes_C D$ and $A' \otimes_{C'} D$ are of constant degree. Then because $K$ is stable under multiplication, it is easy to see that
\begin{equation}
M \cap \phi_0^{-1}\pi_K(L) \subseteq \pi_K(\phi^{-1}L).
\end{equation}
(use the tensor product $-\otimes_A A'$). So $\phi^{-1}L$ is a $J_K$-covering sieve.

(GT3). Let $M$ be a $J_K$-covering sieve and suppose that $h^{-1}L$ is a $J_K$-covering sieve for all $h \in M$. We need to show that $L$ is a $J_K$-covering sieve.

Take $f : C \to D$ with $f \in \pi_K(M)$. The representable numbers on $f$ form a $K$-covering sieve, so take a finitely generated covering sieve $(m_1,\dots,m_k)$ contained in it, with each $m_i$ represented by an Azumaya algebra $A_i$ with center $D$. For the corresponding morphisms $f_i : A \to A_i$ we have that $\pi_K(f_i^{-1}L)$ is a $J$-covering sieve. Moreover, if for each $i \in \{1,\dots,k\}$, $g:D \to E$ is centrally covered w.r.t.\ $f_i^{-1}L$, then $g \circ f$ is centrally covered w.r.t.\ $L$. In other words, we have an inclusion
\begin{equation}
\bigcap_{i=1}^m \pi_K(f_i^{-1}L) \subseteq f^{-1}\pi_K(L)
\end{equation}
and because $\bigcap_{i=1}^m \pi_K(f_i^{-1}L)$ is a $J$-covering sieve, this shows that $f^{-1}\pi_K(L)$ is a $J$-covering sieve too. Moreover, this holds for every $f \in \pi_K(M)$. We conclude that $\pi_K(L)$ is a $J$-covering sieve, i.e.\ $L$ is a $J_K$-covering sieve.
\end{proof}

Suppose that every $K$-covering sieve is nonempty. Let $S = \{C \to C_i \}_{i \in I}$ be a sieve in $\wis{commalg}_{\C}$ and let $\tilde{S}$ be the sieve generated by it in $\wis{Azumaya}_{\C}$. We showed that $S$ is a $J$-covering sieve if and only if $\tilde{S}$ is a $J_K$-covering sieve. So we can recover $J$ from $J_K$. If moreover every $J$-covering sieve is nonempty, then we can similarly recover $K$ from $J_K$.

As concrete examples of Grothendieck topologies on $\wis{Azumaya}_{\C}$, we can take $J$ to be the Zariski or \'etale topology, and $K = K_\Sigma$ one of the topologies from Example~\ref{eg:sigma-topologies}. This gives uncountably many extensions of $J$ to a Grothendieck topology on $\wis{Azumaya}_{\C}$ (one for each choice of $\Sigma$). 
  
 \section{Azumaya representation schemes}
 
 Having constructed Grothendieck topologies on $\wis{Azufine}_{\C}$, let us return to representation stacks and consider the problem for which of these Grothendieck topologies the presheaf functor
  \[
 \wis{Rep}_R~:~\wis{Azufine}_{\C} \rTo \wis{sets} \qquad \mathcal{A} \mapsto \mathrm{Alg}_{\C}(R,A) \]
 which assigns to the affine Azumaya scheme $\mathcal{A}=(\mathrm{Spec}(A),j_A^* \Oscr_A)$ the set $\mathrm{Alg}_{\C}(R,A)$, is actually a sheaf.

 \begin{theorem} \label{main1}
$\wis{Rep}_R$ is a sheaf for any Grothendieck topology on $\wis{Azufine}_{\C}$ coarser than the maximal flat topology.
\end{theorem}

\begin{proof}
Let $S = \{ A \to A_i \}_{i \in I}$ be a covering sieve for the maximal flat topology. Then, by definition, there is a finite subset $J \subseteq I$ such that $A_j$ is of constant degree over its center $Z(A_j)$ for all $j \in J$ and such that the induced morphism $Z(A) \to \prod_{j \in J} Z(A_j)$ is faithfully flat. To show that the $\mathrm{Alg}_{\C}(R,-)$ is a sheaf with respect to $S$, it is enough to prove that
\begin{equation*}
\begin{tikzcd}
0 \ar[r] & \mathrm{Alg}_{\C}(R,A) \ar[r] & \prod_{i\in J}\mathrm{Alg}_{\C}(R,A_i) \ar[r,shift left] \ar[r,shift right] & \prod_{i,j\in J}\mathrm{Alg}_{\C}(R,A_i\otimes_A A_j)
\end{tikzcd}
\end{equation*}
is exact. We know that $\mathrm{Alg}_{\C}(R,-)$ commutes with limits of rings (in particular with categorical kernels and products), so it is enough to show that
\begin{equation*}
\begin{tikzcd}
0 \ar[r] & A \ar[r] & \prod_{i\in J} A_i \ar[r,shift left] \ar[r,shift right] & \prod_{i,j\in J} A_i\otimes_A A_j
\end{tikzcd}
\end{equation*}
is exact. But this follows from the following lemma which we include for lack of a reference.
\end{proof}

\begin{lemma} \label{lmm:ff}
Let $A \to B$ be a morphism in $\wis{Azumaya}_{\C}$. Then the following are equivalent:
\begin{enumerate}
\item $A \to B$ is left faithfully flat; \label{leftff}
\item $A \to B$ is right faithfully flat; \label{rightff}
\item $Z(A) \to Z(B)$ is faithfully flat; \label{centerff}
\end{enumerate}
Moreover, if any of the above is satisfied, then the sequence
\[
\begin{tikzcd}
0 \ar[r] & A \ar[r] & B \ar[r,shift left] \ar[r,shift right] & B \otimes_A B
\end{tikzcd}
\]
is exact.
\end{lemma}
\begin{proof}
(1) $\Leftrightarrow$ (3): by the double centralizer theorem, the functor $- \otimes_A B$ is equivalent to $- \otimes_{Z(A)} Z(B) \otimes_{Z(B)} B^A$. Because $B^A$ is always faithfully flat over its center $Z(B)$, we get that $A \to B$ is left faithfully flat if and only if $Z(A) \to Z(B)$ is faithfully flat.
(2) $\Leftrightarrow$ (3): analogously.

The sequence in the lemma appeared in \cite{Artin1969} and is a non-commutative version of the Amitsur complex. By faithfully flatness, it is enough to check that
\begin{equation}
\begin{tikzcd}
0 \ar[r] & B \ar[r,"{b \mapsto b \otimes 1}"] & B \otimes_A B \ar[r,shift left,"{b\otimes b' \mapsto b \otimes b' \otimes 1}"] \ar[r,shift right,"{b\otimes b' \mapsto b \otimes 1 \otimes b'}"'] & B \otimes_A B \otimes_A B
\end{tikzcd}
\end{equation}
is exact. The morphism $B \to B \otimes_A B$ has a retraction given by the multiplication morphism. In particular it is injective. Further, suppose $\sum_i b_i \otimes b'_i \otimes 1 = \sum_i b_i \otimes 1 \otimes b'_i$. Applying multiplication to the first two tensor factors, we get that $\sum b_i b'_i \otimes 1 = \sum_i b_i \otimes b'_i$. But this means that $\sum_i b_i \otimes b'_i$ lies in the image of $B \to B\otimes_A B$.
\end{proof}

In particular, $\wis{Rep}_R$ will be a sheaf with respect to all Grothendieck topologies $J_K$ on $\wis{Azufine}_{\C}$ where $J$ is the Zariski or \'etale topology on $\wis{affine}_{\C}$ and $K=K_{\Sigma}$ for any subset $\Sigma \subseteq \mathbb{P}$ of prime numbers, as in Example \ref{eg:sigma-topologies}.

\vskip 3mm

Let us now consider a fixed Azumaya algebra $A$ with center $C$. Let $\wis{Azumaya}_A$ be the comma category of all Azumaya algebras $B$ with a morphism $A \rTo B$ in $\wis{Azumaya}_{\C}$ and with morphisms $B \rTo B'$ making the triangle
\begin{equation*}
\begin{tikzcd}[row sep=tiny]
 & B \ar[dd] \\
A \ar[ru] \ar[rd] & \\
& B'
\end{tikzcd}
\end{equation*}
commute. Grothendieck topologies restrict to comma categories and we can consider the composition of functors (for example with respect to the maximal flat topology on $\wis{Azumaya}_{\C}$) between the sheaf toposes
\[
\wis{Sh}(\wis{Azufine}_{\C}) \rTo \wis{Sh}(\wis{Azufine}_A) \rTo \wis{Sh}(\wis{Azufine}_C) \rTo \wis{Sh}(\wis{affine}_C) \]
where the middle arrow is given by $F \mapsto F(A \otimes_C -)$ and the others are given by restriction. As a consequence we have

\begin{proposition} The image of the sheaf $\wis{Rep}_R$ along this composition gives the contravariant functor
\[
\wis{Rep}_A(R)~:~\wis{affine}_C \rTo \wis{sets} \qquad Z \mapsto \mathrm{Alg}_{\C}(R,A \otimes_C \Oscr(Z)) \]
and is therefore a sheaf for the flat topology on $\wis{affine}_C$, and consequently also for the \'etale and the Zariski topology.
\end{proposition}

Our next aim is to prove that this functor is representable by an affine $C$-scheme which we will call the {\em Azumaya representation scheme} of $R$ with respect to the $C$-Azumaya algebra $A$.

For any $\C$-algebra $S$, Artin $S$-bimodules (see \cite{Artin1969} or \cite{ProcesiBook}) are vector spaces $M$ equipped with compatible left and right $S$-action, and generated by invariants $M^S$ as a two-sided $S$-module. Artin $S$-algebras are algebras $R$ equipped with a structure morphism $\phi_R: S \to R$ making $R$ into an Artin bimodule. Equivalently, $\phi_R$ is a Procesi extension \cite{ProcesiBook}.
We will denote by $\wis{Bimod}_S$ the category of Artin $S$-bimodules with morphisms that are $S$-linear on both sides. Similarly, $\wis{Alg}_S$ will denote the category of Artin $S$-algebras with $S$-linear algebra morphisms.

Now let $C$ be a commutative algebra and $A$ an Azumaya algebra over $C$. Note that this makes $A$ into an Artin $C$-algebra. In \cite{Artin1969} it is shown that there are equivalences of categories
\begin{equation} \label{eq:bimodC-bimodA}
\begin{tikzcd}
\wis{Bimod}_C \ar[r,"{A \otimes_C -}", bend left] & \wis{Bimod}_A \ar[l,"{(-)^A}",bend left]
\end{tikzcd},
\end{equation}
\begin{equation} \label{eq:algC-algA}
\begin{tikzcd}
\wis{Alg}_C \ar[r,"{A \otimes_C -}", bend left] & \wis{Alg}_A \ar[l,"{(-)^A}",bend left]
\end{tikzcd}.
\end{equation}
Observe that in the case of an Azumaya algebra $A$ we can reformulate Artin's definition, by invoking the double centralizer theorem. For an Azumaya $A$ with center $C$, Artin $A$-bimodules are the ones such that the induced $C$-action is symmetric. Similarly, Artin $A$-algebras are the algebras with structure morphism sending $C$ into the center.

In order to describe the functor $\wis{Rep}_R(A)$, we have to introduce a generalization of the root algebra $\sqrt[n]{R}$, used in studying $n$-dimensional representations of $R$, see \cite{Bergman} or \cite{Schofield}. Note that morphisms $R \to A$ with $A$ an Azumaya over $C$ are the same as $C$-algebra morphisms $R \otimes C \to A$, so we may assume that $R$ is a $C$-algebra.

\begin{definition}
Let $A$ be an Azumaya algebra with center $C$ and let $R$ be a $C$-algebra. Then the \emph{$A$-th root algebra} of $R$, denoted $\sqrt[A]{R}$, is defined to be
\begin{equation}
\sqrt[A]{R} = (R \ast_C A)^A.
\end{equation}
Here $\ast_C$ denotes the coproduct of $C$-algebras, i.e.\ the pushout of the diagram
\begin{equation*}
\begin{tikzcd}
C \ar[r] \ar[d] & R \\
A &
\end{tikzcd}
\end{equation*}
in the category of rings.
\end{definition}

\begin{proposition} \label{prop:adjoint-root-tensor}
The functor $\sqrt[A]{-} : \wis{Alg}_C \to \wis{Alg}_C$ is left adjoint to tensoring $-\otimes_C A : \wis{Alg}_C \to \wis{Alg}_C$.
\end{proposition}
\begin{proof}
Note that we can write the functor $A\otimes_C - : \wis{Alg}_C \to \wis{Alg}_C$ as a composition
\begin{equation}
\begin{tikzcd}
\wis{Alg}_C \ar[r,"{A \otimes_C -}"] & \wis{Alg}_A \ar[r] & \wis{Alg}_C,
\end{tikzcd}
\end{equation}
where the first functor is the equivalence (\ref{eq:algC-algA}) and the second functor is the forgetful one. Being an equivalence, the first one has its quasi-inverse $(-)^A$ as left adjoint. Further, one can check that the second one has left adjoint $A \ast_C -$. The proposition follows from composition of adjunctions.
\end{proof}

\begin{proposition} \label{eig:root-alg}
Let $A$ and $B$ be Azumaya algebras with center $C$. Let $R$ be a $C$-algebra and $S$ a $\C$-algebra.
\begin{enumerate}
\item There are natural isomorphisms $\sqrt[A]{\sqrt[B]{R}} \simeq \sqrt[A \otimes_C B]{R} \simeq \sqrt[B]{\sqrt[A]{R}}$ of $C$-algebras.
\item For any morphism of commutative algebras $C \to D$, we get natural isomorphisms $\sqrt[A \otimes_C D]{R \otimes_C D} \simeq \sqrt[A]{R} \otimes_C D$.
\item Suppose that $A$ is of constant degree $n$. Then $\sqrt[A]{S \otimes C}$ is, \'etale locally on $C$, isomorphic to $\sqrt[n]{S} \otimes C$. \label{enum:etale-locally}
\item A $C$-linear morphism $A \to B$ induces a $C$-linear morphism $\sqrt[B]{R} \to \sqrt[A]{R}$, functorial in $R$.
\end{enumerate}
\end{proposition}
\begin{proof}
All statements follow by invoking the Yoneda Lemma and some computations. We prove (1) as an example. For any $C$-algebra $S$, we have
\begin{align*}
\wis{Alg}_C(\sqrt[A]{\sqrt[B]{R}}, S) &\simeq \wis{Alg}_C(\sqrt[B]{R}, A \otimes_C S) \\
                                                  &\simeq \wis{Alg}_C(R, B \otimes_C A \otimes_C S) \\
                                                  &\simeq \wis{Alg}_C(\sqrt[A \otimes_C B]{R}, S),
\end{align*}
so by the Yoneda Lemma we have $\sqrt[A]{\sqrt[B]{R}} \simeq \sqrt[A \otimes_C B]{R}$. Similarly for $\sqrt[B]{\sqrt[A]{R}}$.
\end{proof}

\begin{theorem} \label{main2} If $A$ is a degree $n$ Azumaya algebra with center $C$, then for every algebra $R$ there is an affine $C$-scheme $\wis{rep}_A(R)$, which we call the {\em Azumaya representation scheme} of $R$ with respect to $A$, representing the functor
\[
\wis{Rep}_A(R)~:~\wis{affine}_C \rTo \wis{sets} \qquad Z \mapsto \wis{Alg}_{\C}(R,A \otimes_C \Oscr(Z)) \]
\end{theorem}

\begin{proof}
Define the coordinate ring of the Azumaya representation scheme as
\begin{equation*}
\Oscr(\wis{rep}_A(R)) = (\!\sqrt[A]{R \otimes C})_\mathrm{ab}.
\end{equation*}
To check that this represents the given functor, use Proposition \ref{prop:adjoint-root-tensor} and the fact that $- \otimes C$ and the Abelianization functor $\mathrm{ab}$ are both adjoint to the appropriate forgetful functors.
\end{proof}

 \end{document}